\documentclass[11pt]{article}
\usepackage{amscd}
\usepackage{amsfonts}
\usepackage{amsmath}
\usepackage{amssymb}
\usepackage{amsthm}
\usepackage{bbm}
\usepackage{enumerate} 
\usepackage{fancyhdr}
\usepackage{graphicx}
\usepackage{hyperref}
\usepackage{indentfirst}
\usepackage{latexsym}
\usepackage{mathrsfs}
\usepackage[american]{babel}
\usepackage{microtype}

\def\a{\alpha}
\def\b{\beta}
\def\D{\Delta}
\def\d{\delta}

\def\e{\eta}
\def\g{\gamma}
\def\o{\otimes}
\def\tr{\triangleright}
\def\v{\varepsilon}
\def\End{{\rm End}}
\def\Hom{{\rm Hom}}

\allowdisplaybreaks[4]
\newtheorem{theorem}{Theorem}[section]
\newtheorem{lemma}[theorem]{Lemma}
\newtheorem{definition}[theorem]{Definition}
\newtheorem{proposition}[theorem]{Proposition}
\newtheorem{example}[theorem]{Example}

\newtheorem{corollary}[theorem]{Corollary}
\newtheorem{remark}[theorem]{Remark}

\usepackage[top=1in,bottom=1in,left=1.25in,right=1.25in]{geometry}
\date{}
\begin{document}
\renewcommand{\baselinestretch}{1.2}
\renewcommand{\arraystretch}{1.0}
\title{\bf  Post-Hopf group algebras, Hopf group braces and Rota-Baxter operators on Hopf group algebras}

\author {{\bf Yan Ning, Xing Wang\footnote {Corresponding author: xwang17@126.com}, Daowei Lu}\\
{\small School of Mathematics and Big data, Jining University}\\
{\small Qufu, Shandong 273155, P. R. China}
}
 \maketitle

\begin{center}
\begin{minipage}{12.cm}

\noindent{\bf Abstract.} In this paper, we introduce the notions of Hopf group braces, post-Hopf group algebras and Rota-Baxter Hopf group algebras as important generalizations of Hopf brace, post Hopf algebra and Rota-Baxter Hopf algebras respectively. We also discuss their relationships. Explicitly under the condition of cocomutativity, Hopf group braces, post-Hopf group algebras could be mutually obtained, and Rota-Baxter Hopf group algebras could lead to Hopf group braces.
\\

\\

\end{minipage}
\end{center}
\normalsize\vskip1cm

\section{Introduction and preliminaries}
\def\theequation{1.\arabic{equation}}
\setcounter{equation} {0}

To investigate the structure of set-theoretic solutions, Rump introduced braces for abelian groups in \cite{R} , yielding involutive non-degenerate solutions. This framework was later extended to non-abelian groups by Guarnieri and Vendram in \cite{Gua}, who introduced skew braces, providing non-degenerate set-theoretic solutions to the Yang-Baxter equation. Subsequent work by Gateva-Ivanova \cite{Ga} explored these solutions using braided groups and brace structures. More recently, Angiono, Galindo, and Vendramin \cite{AGV} introduced Hopf braces, a unifying generalization of both Rump’s braces and Guarnieri-Vendramin’s skew braces, demonstrating that every Hopf brace induces a solution to the Yang-Baxter equation.

In this paper, we will introduce the notions of Hopf group braces, post-Hopf group algebras and Rota-Baxter Hopf group algebras as important generalizations of Hopf brace, post Hopf algebra and Rota-Baxter Hopf algebras respectively and study their relationships. 

The paper is organized as follows. In Section 2, we firstly introduce the concept of Hopf group braces and study relative properties. we prove that in the case of cocommutative Hopf $\pi$-algebra with $\pi$ being abelian,  Hopf group braces and matched pair of Hopf group algeba are are in one-to-one correspondence. In section 3, we introduce the notion of post-Hopf group algebras and show that a cocommutative post-Hopf group algebra gives rise to a subadjacent Hopf  group algebra together with a module bialgebra structure on the unit-graded part. Then we show that there is a one to-one correspondence between cocommutative post-Hopf  group algebras and cocommutative Hopf group braces. In section 4, we introduce the notion of Rota-Baxter Hopf group algebra, and show that there exists a Hopf group brace structure on it.

Let $\pi$ be an abelian group with the unit element $e$. Recall from \cite{Wang} that a $\pi$-algebra is a family $H=\{H_\a\}_{\a\in\pi}$ of $\Bbbk$-spaces together with a family of $\Bbbk$-linear maps $m=\{m_{\a,\b}:H_\a\o H_\b\rightarrow H_{\a\b}\}_{\a,\b\in\pi}$  (called a multiplication) and a $\Bbbk$-linear map $\eta:k\rightarrow H_e$ (called a unit), such that $m$ is associative in the sense that, for all $\a,\b,\g\in\pi$,
\begin{align*}
&m_{\a,\b\g}(id_{H_\a}\o m_{\b,\g})=m_{\a\b,\g}(m_{\a,\b}\o id_{H_\g}),\\
&m_{\a,e}(id_{H_\a}\o \eta)=id_{H_\a}=m_{e,\a}(\eta\o id_{H_\a}).
\end{align*}

A Hopf $\pi$-algebra is a $\pi$-algebra $H=\{H_\a\}_{\a\in\pi}$ such that for all $\a,\b\in\pi$,
\begin{enumerate}[\quad\rm(1)]
\item each $H_\a$ is a coalgebra with comultiplication $\Delta_\a$ and counit $\v _\a$;
\item $\eta$ and $m_{\a,\b}$ are coalgebras maps;
\item there exists a family of $\Bbbk$-linear maps $S=\{S_\a:H_\a\rightarrow H_{\a^{-1}}\}_{\a\in\pi}$ satisfying 
  $$m_{\a^{-1},\a}(S_\a\o id_{H_{\a}})\Delta_\a=\v _\a1_e=m_{\a,\a^{-1}}(id_{H_{\a}}\o S_\a)\Delta_\a.$$
\end{enumerate}

Let $H$ be a Hopf $\pi$-algebra, then we have the following identities:
\begin{align*}
&S_\b(b)S_\a(a)=S_{\a\b}(ab),\quad S_e(1_e)=1_e,\\
&\Delta_{\a^{-1}}S_\a=\sigma_{H_{\a^{-1}},H_{\a^{-1}}}(S_\a\o S_\a)\Delta_\a,\quad \v _{\a^{-1}}S_\a=\v _\a,
\end{align*}
for all $a\in H_\a,b\in H_\b$, where $\sigma$ denotes the flip map.

When $H$ be a cocommutative Hopf $\pi$-algebra, then we have the identity:
\begin{equation*}
  S_{\a^{-1}}S_\a=id_\a, \quad  \forall\a\in\pi.
\end{equation*}

Let $H$ and $K$ be two Hopf $\pi$-algebras. A morphism of Hopf $\pi$-algebras $f:H\rightarrow K$ is a family of linear maps $f_\a:H_\a\rightarrow K_\a$ satisfying for all $\a,\b\in\pi$
$$m_{\a,\b}(f_\a\o f_\b)=f_{\a\b}m_{\a,\b},\ \Delta_\a f_\a=(f_\a\o f_\a)\Delta_\a,\ S_{K,\a} f_\a=f_{\bar{\a}}S_{H,\a}.$$

Throughout this paper, we will use the Sweedler’s notation: for all $\a\in\pi$ and $h\in H_\a$,
$$\Delta_\a(h)=h_{(1,\a)}\o h_{(2,\a)}.$$
And for $\a\in\pi$, we will denote $\a^{-1}$ by $\bar{\a}$.

\section{Hopf group braces and matched pair of Hopf group algebras}
 \def\theequation{2.\arabic{equation}}
 \setcounter{equation} {0}

\subsection{Hopf group braces}
\begin{definition}
A {\bf Hopf $\pi$-brace} structure over a family of coalgebras $H=\{H_\a,\Delta_\a,\varepsilon_\a\}_{\a\in\pi}$ consisting of the following datum:
\begin{enumerate}[\quad\rm(1)]
\item a Hopf $\pi$-algebra structure $(H, \cdot, 1, S)$ ($(H,\cdot,S)$ or $H$ for short),
\item a Hopf $\pi$-algebra structure $(H, \circ, 1_{\circ}, T)$ ($(H,\circ,T)$ or $H_{\circ}$ for short),
\item the compatible condition
\begin{equation} \label{eq:Hbrace}
g\circ(h\ell)=(g_{(1,\a)}\circ h)S_{\a}(g_{(2,\a)})(g_{(3,\a)}\circ \ell),
\end{equation}
for all $g\in H_\a,h\in H_\b,\ell\in H_\g, ~\a,\b,\g\in\pi$.
\end{enumerate}
\end{definition}
It is denoted by $(H,\cdot,1, S, \circ, 1_{\circ}, T)$ or $(H,\cdot,\circ)$ for short.

\begin{remark}
\begin{enumerate}[\rm(1)]
  \item When the group $\pi=\{1\}$, we could recover the notion of Hopf braces.
  \item In any Hopf $\pi$-brace, $1_\circ=1$. Indeed, setting $g=h=1_\circ$ in \eqref{eq:Hbrace} one obtains $1_\circ \ell=\ell$ for all $\ell\in H_\g$. Similarly, $g=\ell=1_\circ$ yields $h1_\circ=h$ for all $h\in H_\b$.
  \item  $(H,\cdot,\circ)$ is called cocommutative if $H_\a$ is cocommutative for each $\a\in\pi$.
\end{enumerate}
\end{remark}

A homomorphism $f:(H,\cdot_H,\circ_H)\rightarrow(K,\cdot_K,\circ_K)$ of Hopf $\pi$-braces is a linear map such that $f:H\rightarrow K$ and $f:H_{\circ}\rightarrow K_{\circ}$ are homomorphisms of Hopf $\pi$-algebras. 

Fix a Hopf $\pi$-algebra $(H,\cdot,1,\Delta,\v,S)$, let $\mathbf{Br}(H)$ denote the full subcategory of the category of Hopf $\pi$-braces with objects $(H,\cdot,\circ)$, that is, all objects in $\mathbf{Br}(H)$ share the same Hopf algebra structure $(H,\cdot)$.

\begin{example}
Let $(H,\cdot,\circ,S,T)$ be a Hopf brace and $\pi=Aut(H)$ the group of automorphisms of Hopf brace $H$. For each $\a\in\pi$, $H_\a=H$ as a vector space. We denote the element in $H_\a$ by $h^\a=\a(h)$ for $h\in H$. Define
\begin{align*}
&\cdot_{\a,\b}:H_\a\o H_\b\rightarrow H_{\a\b},\ g^\a\o h^\b\mapsto (g \cdot h)^{\a\b},\\
&\circ_{\a,\b}:H_\a\o H_\b\rightarrow H_{\a\b},\ g^\a\o h^\b\mapsto (g \circ h)^{\a\b},\\
&1_\cdot=1_\circ=1,\quad \Delta_\a=\Delta,\quad \varepsilon_\a=\varepsilon,\\
&S_\a:H_\a\rightarrow H_{\bar{\a}},\ h^\a\mapsto S(h)^{\bar{\a}},\\
&T_\a:H_\a\rightarrow H_{\bar{\a}},\ h^\a\mapsto T(h)^{\bar{\a}}.
\end{align*}
Denote $\cdot=\{\cdot_{\a,\b}\}_{\a,\b\in\pi},\circ=\{\circ_{\a,\b}\}_{\a,\b\in\pi},S=\{S_\a\}_{\a\in\pi},T=\{T_\a\}_{\a\in\pi})$, then $(H,\cdot,S,\circ,T$ is a Hopf $\pi$-brace.
\end{example}

\begin{definition}
Let $H=\{H_\a\}_{\a\in\pi}$ be a Hopf $\pi$-algebra, and $K$ a Hopf algebra. Then $H$ is called a {\bf left $K$-$\pi$-module} with the action $\rightharpoondown=\{\rightharpoondown_\a: K\otimes H_\a \rightarrow H_\a\}_{\a\in\pi}$ 
if every $H_\a$ is a $H$-module. And $H$ is called a {\bf left $K$-$\pi$-module $\pi$-bialgebra} if the following conditions hold:
\begin{enumerate}[\quad\rm(1)]
  \item $k\rightharpoondown_{\a\b}(hg)=(k_{(1)}\rightharpoondown_{\a}h)(k_{(2)}\rightharpoondown_{\b}h)$,
  \item $k\rightharpoondown_{e}1_e=\v_K(k)1_e$,
  \item $(k\rightharpoondown_{\a}h)_{(1,\a)} \otimes (k\rightharpoondown_{\a}h)_{(2,\a)} = (k_{(1)}\rightharpoondown_{\a}h_{(1,\a)})\otimes(k_{(2)}\rightharpoondown_{\a}h_{(2,\a)})$,
  \item $\v_\a(k\rightharpoondown_\a h)=\v_K(k)\v_\a(h)$.
\end{enumerate}
for all $h\in H_\a, g\in H_\b, \a,\b\in\pi$, $k\in K$.
\end{definition}

\begin{example} \label{ex:pimod}
Let $H=\{H_\a\}_{\a\in\pi}$ be a cocommutative Hopf $\pi$-algebra, $K$ a cocommutative Hopf algebra, 
and $H$ a left $K$-module  $\pi$-bialgebra as above. 
Then the smash product $(H\sharp K=\{H_\a\sharp K\}_{\a\in\pi},\cdot,\circ)$ is a cocommutative Hopf $\pi$-brace with the following structures:
\begin{eqnarray*}
  (h\sharp k)\cdot_{\a\b}(h'\sharp k') &=& hh'\sharp kk', \\
  (h\sharp k)\circ_{\a\b}(h'\sharp k') &=& h(k_{(1)}\rightharpoondown_{\b}h')\sharp k_{(2)}k', \\
  S_{H_\a\sharp K}(h\sharp k) &=& S_\a(h)\sharp S_K(k), \\
  T_{H_\a\sharp K}(h\sharp k) &=& S_K(k_{(1)})\rightharpoondown_{\bar{\a}}S_\a(h)\sharp S_K(k_{(2)}), 
\end{eqnarray*}
for all $h\in H_\a, h'\in H_\b, \a,\b\in\pi$, $k,k'\in K$.
\end{example}

\begin{definition}\cite{Wang}
Let $K=\{K_\a\}_{\a\in\pi}$ a Hopf $\pi_1$-algebra. Then a vector space $M$ is called a  left $\pi$-$K$-modulelike object with the action $\rightharpoonup=\{\rightharpoonup_\a: K_\a \otimes M \rightarrow M\}_{\a\in\pi_1}$ if 
$$(k\ell)\rightharpoonup_{\a\b} m = k\rightharpoonup_\a(\ell\rightharpoonup_\b m)\ \hbox{and}\ 1_e \rightharpoonup_e m =m,$$ 
for all $m\in M$, $k\in K_\a,\ell\in K_\b, \a,\b\in\pi_1$. 

Now let $H$ be a Hopf $\pi_2$-algebra for group $\pi_2$ and each $H_\g$ is a left $\pi$-$K$-modulelike object, then $H$ is called a  left $\pi$-$K$-modulelike bialgebra if the following conditions hold:
\begin{enumerate}[\quad\rm(1)]
  \item $k\rightharpoonup_{\a}(hg)=(k_{(1,\a)}\rightharpoonup_\a h)(k_{(2,\a)}\rightharpoonup_\a g)$,
  \item $k\rightharpoonup_{\a}1_H=\v_\a(k)1_H$,
  \item $(k\rightharpoonup_\a h)_{(1,\g)} \otimes (k\rightharpoonup_\a h)_{(2,\g)} = (k_{(1,\a)}\rightharpoonup_\a h_{(1,\g)})\otimes(k_{(2,\a)}\rightharpoonup_\a h_{(2,\g)})$,
  \item $\v(k\rightharpoonup_\a h)=\v_{K,\a}(k)\v_{H,\g}(h)$,
\end{enumerate}
for all $h\in H_\g,g\in H_\d$, $k\in K_\a$.
\end{definition}

\begin{example} \label{ex:modlike}
Let $K$ be a cocommutative Hopf $\pi_1$-algebra and $H$ a cocommutative left $\pi$-$K$-modulelike Hopf $\pi_2$-algebra. Denote $\pi=\pi_1\bigoplus\pi_2$, then the smash product $(H\natural K=\{H_\g\sharp K_\a\}_{(\g,\a)\in\pi},\cdot,\circ)$ is a cocommutative Hopf $\pi$-brace with the following structures:
\begin{eqnarray*}
  (h\natural k)\cdot(h'\natural k') &=& hh'\natural kk', \\
  (h\natural k)\circ(h'\natural k') &=& h(k_{(1,\a)}\rightharpoonup_{\a}h')\natural k_{(2,\a)}k', \\
  S_{(\g,\a)}(h\natural k) &=& S_\g(h)\natural S_\a(k), \\
  T_{(\g,\a)}(h\natural k) &=& S_\a(k_{(1,\a)})\rightharpoonup_{\bar{\a}}S_\g(h)\natural S_\a(k_{(2,\a)}), 
\end{eqnarray*}
for all $h\in H_\g,h'\in H_\d$, $k\in K_\a,k'\in K_\b$.
\end{example}

\begin{lemma}	\label{lem:truco}
Let $(H,\cdot,\circ)$ be a Hopf $\pi$-brace. Then
\begin{equation}\label{eq:Hbrace2}
  S_{\a\b}(g_{(1,\a)}\circ h)g_{(2,\a)} = S_{\a}(g_{(1,\a)})(g_{(2,\a)}\circ S_{\b}(h))
\end{equation}
for all $g\in H_\a, h\in H_\b,~ \a,\b\in\pi$. 
\end{lemma}

\begin{proof}
Equation~\eqref{eq:Hbrace} implies that
\begin{equation*}	
\v_\b(h)g = g\circ (h_{(1,\b)}S_{\b}(h_{(2,\b)})) = (g_{(1,\a)}\circ h_{(1,\b)})S_{\a}(g_{(2,\a)})(g_{(3,\a)}\circ S_{\b}(h_{(2,\b)}))
\end{equation*}
holds for all $g\in H_\a, h\in H_\b,~ \a,\b\in\pi$, and hence,
\begin{align*}
S_{\a\b}(g_{(1,\a)}\circ h)g_{(2,\a)} &= S_{\a\b}(g_{(1,\a)}\circ h_{(1,\b)}\v_\b(h_{(2,\b)}))g_{(2,\a)} \\
&= S_{\a\b}(g_{(1,\a)}\circ h_{(1,\b)})\v_\b(h_{(2,\b)})g_{(2,\a)}\\
&=S_{\a\b}(g_{(1,\a)}\circ h_{(1,\b)})(g_{(2,\a)}\circ h_{(2,\b)})S_\a(g_{(3,\a)})(g_{(4,\a)}\circ S_\b(h_{(3,\b)}))\\
&=\v_{\a\b}(g_{(1,\a)}\circ h_{(1,\b)})S_\a(g_{(2,\a)})(g_{(3,\a)}\circ S_\b(h_{(2,\b)}))\\
&=S_\a(g_{(1,\a)})(g_{(2,\a)}\circ S_\b(h)).
\end{align*}
This completes the proof.
\end{proof}

\begin{proposition}  \label{pro:pi-bialg}
Let $(H,\cdot,\circ)$ be a Hopf $\pi$-brace, $H$ denote the Hopf $\pi$-algebra $(H, \cdot, S)$ and $H_{\circ}$ denote the Hopf $\pi$-algebra $(H,\circ,T)$.
\begin{enumerate}[\quad\rm(1)]
\item For all $g\in H_\a, h\in H_\b$, $H$ is a left $H_{\circ}$-$\pi$-module like $\pi$-algebra with
    \begin{equation*} \label{eq:hbm:up}
     g\rightharpoonup_\a h=S_\a(g_{(1,\a)})(g_{(2,\a)}\circ h).
    \end{equation*}
\item For all $g\in H_\a, h\in H_\b$,
    \begin{eqnarray*}
     &g\circ h=g_{(1,\a)}(g_{(2,\a)}\rightharpoonup_\a h),& \label{eq:hbm1}\\
     &gh=g_{(1,\a)}\circ(T_\a(g_{(2,\a)})\rightharpoonup_{\bar{\a}} h).& \label{eq:hbe1}
    \end{eqnarray*}
\item If $H$ is cocommutative, then $H$ is a left $H_{\circ}$-$\pi$-module like $\pi$-bialgebra and
\begin{equation*}
  S_\b(g\rightharpoonup_\a h)=g\rightharpoonup_{\a} S_\b(h), ~g\in H_\a, h\in H_\b.
\end{equation*}
\end{enumerate}
\end{proposition}
\begin{proof}
(1)For $g\in H_\a, h\in H_\b,\ell\in H_\g$, it is clear that, $1\rightharpoonup_\a h=S_e(1)(1\circ h)=h$ and 
$$g\rightharpoonup_\a 1=S_\a(g_{(1,\a)})(g_{(2,\a)}\circ 1)=S_\a(g_{(1,\a)})g_{(2,\a)}=\v_\a(g)1.$$

And
\begin{eqnarray*}
& & g \rightharpoonup_\a (h\rightharpoonup_\b\ell )\\
&=& S_\a(g_{(1,\a)}) (g_{(2,\a)}\circ (S_\b(h_{(1,\b)})(h_{(2,\b)}\circ \ell ))) \\
&\overset{(\ref{eq:Hbrace})}{=}& S_\a(g_{(1,\a)}) (g_{(2,\a)}\circ S_\b(h_{(1,\b)})) S_\a(g_{(3,\a)}) (g_{(4,\a)}\circ h_{(2,\b)}\circ \ell ) \\
&\overset{(\ref{eq:Hbrace2})}{=}& S_{\a\b}(g_{(1,\a)}\circ h_{(1,\b)}) g_{(2,\a)} S_\a(g_{(3,\a)}) (g_{(4,\a)}\circ h_{(2,\b)}\circ \ell ) \\
&=& S_{\a}(g_{(1,\a)}\circ h_{(1,\b)}) ((g_{(2,\a)}\circ h_{(2,\b)})\circ \ell ) \\
&=& (g\circ h)\rightharpoonup_{\a\b} \ell ,
\end{eqnarray*}
and
\begin{eqnarray*}
g\rightharpoonup_{\a}(hk) &=& S_\a(g_{(1,\a)}) (g_{(2,\a)}\circ(hk)) \\
 &=& S_\a(g_{(1,\a)}) (g_{(2,\a)}\circ h)S_\a(g_{(3,\a)}) (g_{(4,\a)}\circ k)\\
 &=& (g_{(1,\a)}\rightharpoonup_\a h)(g_{(2,\a)}\rightharpoonup_\a k).
\end{eqnarray*}
Therefore, $H$ is a left $H_{\circ}$-$\pi$-module $\pi$-algebra.

(2) For all $g\in H_\a, h\in H_\b$,
\begin{eqnarray*}
 g_{(1,\a)}(g_{(2,\a)}\rightharpoonup_\a h) 
 &=& g_{(1,\a)}S_\a(g_{(2,\a)})(g_{(3,\a)}\circ h) \\
 &=& \v_\a(g_{(1,\a)})(g_{(2,\a)}\circ h) = g\circ h,
\end{eqnarray*}
and
\begin{eqnarray*}
 g_{(1,\a)}\circ(T_\a(g_{(2,\a)})\rightharpoonup_{\bar{\a}} h) 
 &=& g_{(1,\a)}(g_{(2,\a)}\rightharpoonup_\a(T_\a(g_{(2,\a)})\rightharpoonup_{\a} h)) \\
 &=& g_{(1,\a)}((g_{(2,\a)}\circ T_\a(g_{(2,\a)}))\rightharpoonup_{e} h) \\
 &=& g_{(1,\a)}\v_\a(g_{(2,\a)})(1\rightharpoonup_{e} h) = gh.
\end{eqnarray*}

(3) If $H$ is cocommutative, it is straightforward to verify that $H$ is a left $H_{\circ}$-$\pi$-module like $\pi$-bialgebra.
For all $g\in H_\a, h\in H_\b$, we have
\begin{eqnarray*}
S_\b(g\rightharpoonup_\a h) &=& S_\b(S_\a(g_{(1,\a)})(g_{(2,\a)}\circ h)) \\
 &=& S_{\a\b}(g_{(1,\a)}\circ h)S_{\bar{\a}}S_\a(g_{(2,\a)})\\
 &=& S_{\a\b}(g_{(1,\a)}\circ h)g_{(2,\a)} \\
 &=& S_\a(g_{(1,\a)})(g_{(2,\a)}\circ S_\b(h)) \\
 &=& g\rightharpoonup_{\a} S_\b(h). 
\end{eqnarray*}
These finish the proof.
\end{proof}

\begin{proposition}
Let $\pi$ be an abelian group and $(H,\cdot,\circ)$ a cocommutative Hopf $\pi$-brace, then each $H_\a$ is a right $H_\circ$-$\pi$-module like coalgebra under the action
$$a\leftharpoonup_\b x=T_\b(a_{(1,\a)}\rightharpoonup_\a x_{(1,\b)})\circ a_{(2,\a)}\circ x_{(2,\b)},$$
for all $a\in H_\a,x\in H_\b$.
\end{proposition}

\begin{proof}
For all $a\in H_\a,x\in H_\b,y\in H_\g$,
\begin{align*}
&(a\leftharpoonup_\b x)\leftharpoonup_\g y\\
&=(T_\b(a_{(1,\a)}\rightharpoonup_\a x_{(1,\b)})\circ a_{(2,\a)}\circ x_{(2,\b)})\leftharpoonup_\g y\\
&=T_{\bar{\g}}((T_\b(a_{(1,\a)}\rightharpoonup_\a x_{(1,\b)})\circ a_{(2,\a)}\circ x_{(2,\b)})\rightharpoonup y_{(1,\g)})\circ T_\b(a_{(3,\a)}\rightharpoonup_\a x_{(3,\b)})\circ a_{(4,\a)}\circ x_{(4,\b)}\circ y_{(2,\g)}\\
&=T_{\b\g}((a_{(1,\a)}\rightharpoonup_\a x_{(1,\b)})\circ(T_\b(a_{(2,\a)}\rightharpoonup_\a x_{(2,\b)})\rightharpoonup_{\a\b}  ((a_{(3,\a)}\circ x_{(3,\b)})\rightharpoonup y_{(1,\g)})))\circ a_{(4,\a)}\circ x_{(4,\b)}\circ y_{(2,\g)}\\
&=T_{\b\g}((a_{(1,\a)}\rightharpoonup_\a x_{(1,\b)}) ((a_{(2,\a)}\circ x_{(2,\b)})\rightharpoonup_{\a\b} y_{(1,\g)})))\circ a_{(3,\a)}\circ x_{(3,\b)}\circ y_{(2,\g)}\\
&=T_{\b\g}(S_\a(a_{(1,\a)})(a_{(2,\a)}\circ x_{(1,\b)})S_{\a\b}(a_{(3,\a)}\circ x_{(2,\b)})(a_{(4,\a)}\circ x_{(3,\b)}\circ y_{(1,\g)}))\circ a_{(5,\a)}\circ x_{(4,\b)}\circ y_{(2,\g)}\\
&=T_{\b\g}(S_\a(a_{(1,\a)})(a_{(2,\a)}\circ x_{(1,\b)}\circ y_{(1,\g)}))\circ a_{(3,\a)}\circ x_{(2,\b)}\circ y_{(2,\g)}\\
&=T_{\b\g}(a_{(1,\a)}\rightharpoonup_\a(x_{(1,\b)}\circ y_{(1,\g)}))\circ a_{(2,\a)}\circ x_{(2,\b)}\circ y_{(2,\g)}\\
&=a\leftharpoonup_{\b\g} (x\circ y),
\end{align*}
and 
\begin{align*}
&\Delta_\a(a\leftharpoonup_\b x)=\Delta_\a(T_\b(a_{(1,\a)}\rightharpoonup_\a x_{(1,\b)})\circ a_{(2,\a)}\circ x_{(2,\b)})\\
&=T_\b(a_{(1,\a)}\rightharpoonup_\a x_{(1,\b)})\circ a_{(2,\a)}\circ x_{(2,\b)}\o T_\b(a_{(3,\a)}\rightharpoonup_\a x_{(3,\b)})\circ a_{(4,\a)}\circ x_{(4,\b)}\\
&=a_{(1,\a)}\leftharpoonup_\b x_{(1,\b)}\o a_{(2,\a)}\leftharpoonup_\b x_{(2,\b)}.
\end{align*}
The proof is completed.
\end{proof}

\begin{theorem}
Let $\pi$ be an abelian group and $(H,\cdot,\circ)$ a cocommutative Hopf $\pi$-brace. Define $c_{\a,\b}:H_\a\o H_\b\rightarrow H_\b\o H_\a$ by
$$c_{\a,\b}(x\o y)=x_{(1,\a)}\rightharpoonup_\a y_{(1,\b)}\o x_{(2,\a)}\leftharpoonup_\b y_{(2,\b)},\ \forall x\in H_\a,y\in H_\b,$$
and $\sigma_{\a,\b}:H_\a\o H_\b\rightarrow H_\b\o H_\a$ by
$$\sigma_{\a,\b}(x\o y)=y_{(1,\a)}\o S_\b(y_{(2,\b)})ay_{(3,\a)},$$
then $c=\{c_{\a,\b}\}$ and $\sigma=\{\sigma_{\a,\b}\}$ produce isomorphic representations of the braid group $B_n$ on $H_{\a_1}\o H_{\a_2}\o\cdots\o H_{\a_n}$ for $n\geq2$ and $\a_i\in\pi$.
\end{theorem}

\begin{proof}
Define $f_n,g_n:H_{\a_1}\o H_{\a_2}\o\cdots\o H_{\a_n}\rightarrow H_{\a_1}\o H_{\a_2}\o\cdots\o H_{\a_n}$ as follows
$$g_n(a^{\a_1}\o a^{\a_2}\o\cdots\o a^{\a_n})=a^{\a_1}_{(1,\a_1)}\o a^{\a_1}_{(2,\a_1)}\rightharpoonup_{\a_1}a^{\a_2}\o\cdots\o a^{\a_1}_{(n,\a_1)}\rightharpoonup_{\a_1}a^{\a_n},$$
$f_2=f:H_\a\o H_\b\rightarrow H_\a\o H_\b$, $f(x\o y)=x_{(1,\a)}\o (x_{(2,\a)}\rightharpoonup_\a y)$ and recursively $f_n=g_n(id\o f_{n-1})$. Since $f$ is invertible with $f^{-1}(x\o y)=x_{(1,\a)}\o (T_\a(x_{(2,\a)})\rightharpoonup_{\bar{\a}}y)$ and the $g_n$ are invertible, it follows that the $g_n$ are invertible by induction.
\end{proof}

\begin{corollary}
then $c=\{c_{\a,\b}\}_{\a,\b\in\pi}$ is a coalgebra isomorphism and a solution of the braid equation.
\end{corollary}

\subsection{Matched pair of Hopf group algebras}

 \begin{definition}
Let $\pi_1$ and $\pi_2$ be two groups. A matched pair of Hopf $\pi$-algebras is a 4-tuple $(H, K,\leftharpoonup,\rightharpoonup)$, where $K$ is Hopf $\pi_1$-algebra and $H$ is Hopf $\pi_2$-algebra, $\rightharpoonup_{\a}:H_\a\o K_\g\rightarrow K_\g$ and $\leftharpoonup_{\g}:H_\a\o K_\g \rightarrow H_\a$ are linear maps such that $K_\b$ is a left $\pi_2$-$H$-modulelike coalgebra and $H_\a$ is a right $\pi_1$-$K$-modulelike coalgebra and the following compatibility
conditions hold:
\begin{align}
x\rightharpoonup_{\g\d}(ab)&=(x_{(1,\a)}\rightharpoonup_{\a}a_{(1,\g)})((x_{(2,\a)}\leftharpoonup_{\g}a_{(2,\g)})\rightharpoonup_\a b)),\label{mp1}\\
x\rightharpoonup_{\a}1_K&=\v_\a(x)1_K,\label{mp2}\\
(xy)\leftharpoonup_{\g}a&=(x\leftharpoonup_{\g}(y_{(1,\b)}\rightharpoonup_{\b}a_{(1,\g)}))(y_{(2,\b)}\leftharpoonup_{\g}a_{(2,\g)}),\label{mp3}\\
1_H\leftharpoonup_e a&=\v_\g(a)1_H,\label{mp4}\\
(x_{(1,\a)}\leftharpoonup_\g a_{(1,\g)})\o(x_{(2,\a)}\rightharpoonup_\a a_{(2,\g)})&=(x_{(2,\a)}\leftharpoonup_\g a_{(2,\g)})\o(x_{(1,\a)}\rightharpoonup_\a a_{(1,\g)}),\label{mp5}
\end{align}
for all $x\in H_\a,y\in H_\b,a\in K_\g,b\in K_\d,\g,\d\in\pi_1,\a,\b\in\pi_2$.
\end{definition}

\begin{theorem}
Let $(H, K,\leftharpoonup,\rightharpoonup)$ be a matched pair of Hopf group algebra and $\pi=\pi_1\bigoplus\pi_2$, then $K\o H=\{(K\o H)_{(\g,\a)}=K_\g\o H_\a\}_{(\g,\a)\in\pi}$ becomes a Hopf $\pi$-algebra with the following structures:
\begin{align*}
&(a\o x)(b\o y)=a(x_{(1,\a)}\rightharpoonup_\a b_{(1,\d)})\o(x_{(2,\a)}\leftharpoonup_\d b_{(2,\d)})y,\\
&(a\o x)_{(1,(\g,\a))}\o (a\o x)_{(2,(\g,\a))}=(a_{(1,\a)}\o x_{(1,\g)})\o (a_{(2,\a)}\o x_{(2,\g)}),\\
&\v_{(\g,\a)}(a\o x)=\v_{K,\g}(a) \v_{H,\a}(x),\\
&S_{(\g,\a)}(a\o x)=(S_{H,\a}(x_{(2,\a)})\rightharpoonup_{\bar{\a}}S_{K,\g}(a_{(2,\g)}))\o(S_{H,\a}(x_{(1,\a)})\leftharpoonup_{\bar{\g}}S_{K,\g}(a_{(1,\g)})),
\end{align*}
for all $x\in H_\a,y\in H_\b,a\in K_\g,b\in K_\d$.
\end{theorem}

\begin{proof}
For all $x\in H_\a,y\in H_\b,z\in H_\mu,a\in K_\g,b\in K_\d,c\in K_\nu$,
\begin{align*}
&[(a\o x)(b\o y)](c\o z)\\
&=[a(x_{(1,\a)}\rightharpoonup_\a b_{(1,\d)})\o(x_{(2,\a)}\leftharpoonup_\d b_{(2,\d)})y](c\o z)\\
&=a(x_{(1,\a)}\rightharpoonup_\a b_{(1,\d)})[((x_{(2,\a)}\leftharpoonup_\d b_{(2,\d)})y)_{(1,\a\b)}\rightharpoonup_{\a\b}c_{(1,\nu)}]\\
&\quad\o[((x_{(2,\a)}\leftharpoonup_\d b_{(2,\d)})y)_{(2,\a\b)}\leftharpoonup_{\nu}c_{(2,\nu)}]z\\
&=a(x_{(1,\a)}\rightharpoonup_\a b_{(1,\d)})[(x_{(2,\a)}\leftharpoonup_\d b_{(2,\d)})y_{(1,\b)}\rightharpoonup_{\a\b}c_{(1,\nu)}]\\
&\quad\o[(x_{(3,\a)}\leftharpoonup_\d b_{(3,\d)})y_{(2,\b)}\leftharpoonup_{\nu}c_{(2,\nu)}]z\\
&=a(x_{(1,\a)}\rightharpoonup_\a b_{(1,\d)})[(x_{(2,\a)}\leftharpoonup_\d b_{(2,\d)})\rightharpoonup_{\a}(y_{(1,\b)}\rightharpoonup_{\b}c_{(1,\nu)})]\\
&\quad\o[((x_{(3,\a)}\leftharpoonup_\b b_{(3,\b)})\leftharpoonup_\nu (y_{(2,\b)}\rightharpoonup_{\b}c_{(2,\nu)}))(y_{(3,\b)}\leftharpoonup_{\nu}c_{(3,\nu)})]z\\
&=a(x_{(1,\a)}\rightharpoonup_{\a} b_{(1,\d)}(y_{(1,\b)}\rightharpoonup_{\b}c_{(1,\nu)}))\\
&\quad\o[(x_{(2,\a)}\leftharpoonup_{\d\nu} b_{(2,\d)} (y_{(2,\b)}\rightharpoonup_{\b}c_{(2,\nu)}))(y_{(2,\b)}\leftharpoonup_{\nu}c_{(2,\nu)})]z\\
&=(a\o x)(b(y_{(1,\b)}\rightharpoonup_\b c_{(1,\nu)})\o(y_{(2,\b)}\leftharpoonup_\nu c_{(2,\nu)})z)\\
&=(a\o x)[(b\o y)(c\o z)],
\end{align*}
thus $K\o H$ is associative. 
\begin{align*}
&\Delta_{(\g\d,\a\b)}((a\o x)(b\o y))\\
&=\Delta_{(\g\d,\a\b)}(a(x_{(1,\a)}\rightharpoonup_\a b_{(1,\d)})\o(x_{(2,\a)}\leftharpoonup_\d b_{(2,\d)})y)\\
&=(a(x_{(1,\a)}\rightharpoonup_\a b_{(1,\d)}))_{(1,\g\d)}\o((x_{(2,\a)}\leftharpoonup_\d b_{(2,\d)})y)_{(1,\a\b)}\\
&\quad \o(a(x_{(1,\a)}\rightharpoonup_\a b_{(1,\d)}))_{(2,\g\d)}\o((x_{(2,\a)}\leftharpoonup_\d b_{(2,\d)})y)_{(2,\a\b)}\\
&=a_{(1,\g)}(x_{(1,\a)}\rightharpoonup_\a b_{(1,\d)})_{(1,\d)}\o(x_{(2,\a)}\leftharpoonup_\d b_{(2,\d)})_{(1,\a)}y_{(1,\b)}\\
&\quad \o a_{(2,\g)}(x_{(1,\a)}\rightharpoonup_\a b_{(1,\d)})_{(2,\d)}\o(x_{(2,\a)}\leftharpoonup_\d b_{(2,\d)})_{(2,\a)}y_{(2,\b)}\\
&=a_{(1,\g)}(x_{(1,\a)}\rightharpoonup_\a b_{(1,\d)})\o(x_{(3,\a)}\leftharpoonup_\b b_{(3,\b)})y_{(1,\b)}\\
&\quad \o a_{(2,\g)}(x_{(2,\a)}\rightharpoonup_\a b_{(2,\d)})\o(x_{(4,\a)}\leftharpoonup_\b b_{(4,\b)})y_{(2,\b)}\\
&=a_{(1,\g)}(x_{(1,\a)}\rightharpoonup_\a b_{(1,\d)})\o (x_{(2,\a)}\rightharpoonup_\a b_{(2,\d)})y_{(1,\b)}\\
&\quad \o a_{(2,\g)}(x_{(3,\a)}\leftharpoonup_\d b_{(3,\d)})\o(x_{(4,\a)}\leftharpoonup_\d b_{(4,\d)})y_{(2,\b)}\\
&=\Delta_{(\g,\a)}(a\o x)\Delta_{(\d,\b)}(b\o y),
\end{align*}
and
\begin{align*}
&S_{(\g,\a)}(a_{(1,\g)}\o x_{(1,\a)})(a_{(2,\g)}\o x_{(2,\a)})\\
&=(S_{H,\a}(x_{(2,\a)})\rightharpoonup_{\bar{\a}}S_{K,\g}(a_{(2,\g)})\o S_{H,\a}(x_{(1,\a)})\leftharpoonup_{\bar{\g}}S_{K,\g}(a_{(1,\g)}))(a_{(3,\g)}\o x_{(3,\a)})\\
&=(S_{H,\a}(x_{(2,\a)})\rightharpoonup_{\bar{\a}}S_{K,\g}(a_{(2,\g)}))((S_{H,\a}(x_{(1,\a)})\leftharpoonup_{\bar{\g}}S_{K,\g}(a_{(1,\g)}))_{(1,\bar{\a})}\rightharpoonup_{\bar{\a}} a_{(3,\g)})\\
&\quad\o((S_{H,\a}(x_{(1,\a)})\leftharpoonup_{\bar{\g}}S_{K,\g}(a_{(1,\g)}))_{(2,\bar{\a})}\leftharpoonup_{\a} a_{(4,\g)})x_{(3,\a)}\\
&=(S_{H,\a}(x_{(2,\a)})\rightharpoonup_{\bar{\a}}S_{K,\g}(a_{(2,\g)}))((S_{H,\a}(x_{(1,\a)})_{(1,\bar{\a})}\leftharpoonup_{\bar{\g}}S_{K,\g}(a_{(1,\g)})_{(1,\bar{\a})})\rightharpoonup_{\a} a_{(3,\g)})\\
&\quad\o((S_{H,\a}(x_{(1,\a)})_{(2,\bar{\a})}\leftharpoonup_{\bar{\g}}S_{K,\g}(a_{(1,\g)})_{(2,\bar{\a})})\leftharpoonup_{\a} a_{(4,\g)})x_{(3,\a)}\\
&=(S_{H,\a}(x_{(3,\a)})\rightharpoonup_{\bar{\a}}S_{K,\g}(a_{(3,\g)}))((S_{H,\a}(x_{(2,\a)})\leftharpoonup_{\bar{\g}}S_{K,\g}(a_{(2,\g)}))\rightharpoonup_{\a} a_{(4,\g)})\\
&\quad\o((S_{H,\a}(x_{(1,\a)})\leftharpoonup_{\bar{\g}}S_{K,\g}(a_{(1,\g)}))\leftharpoonup_{\a} a_{(5,\a)})x_{(4,\a)}\\
&=S_{H,\a}(x_{(2,\a)})\rightharpoonup_{\bar{\a}}(S_{K,\g}(a_{(3,\g)})a_{(4,\g)})\o((S_{H,\a}(x_{(1,\a)})\leftharpoonup_{\bar{\g}}S_{K,\g}(a_{(1,\g)}))\leftharpoonup_{\a} a_{(5,\a)})x_{(3,\a)}\\
&=((S_{H,\a}(x_{(1,\a)})\leftharpoonup_{\bar{\g}}S_{K,\g}(a_{(1,\g)})a_{(2,\g)})x_{(2,\a)}\\
&=\v_{K,\a}(a) \v_{H,\a}(x).
\end{align*}
Similarly $(a_{(1,\g)}\o x_{(1,\a)})S_{\g,\a}(a_{(2,\g)}\o x_{(2,\a)})=\v_{K,\g}(a) \v_{H,\a}(x).$
\end{proof}

\begin{proposition}\label{pro:brace}
Let $\pi$ be an abelian group and  $(H,\cdot,\circ)$ be a cocommutative Hopf $\pi$-brace, then $(H_\circ,H_\circ)$ is a matched pair of cocommutative Hopf $\pi$-braces with 
$$x\rightharpoonup_\a a=S_\a(x_{(1,\a)})(x_{(2,\a)}\circ a),\ x\leftharpoonup_\b a=T_\b(x_{(1,\a)}\rightharpoonup_\b a_{(1,\b)})\circ x_{(2,\a)}\circ a_{(2,\b)},$$
for all $x\in H_\a,a\in H_\b$.
\end{proposition}

\begin{proof}
For all $x\in H_\a,y\in H_\b,a\in H_\g,b\in H_\d$,
\begin{align*}
&(x_{(1,\a)}\rightharpoonup_{\a}a_{(1,\g)})\circ((x_{(2,\a)}\leftharpoonup_{\g}a_{(2,\g)})\rightharpoonup_\a b))\\
&=(x_{(1,\a)}\rightharpoonup_{\a}a_{(1,\g)})_{(1,\g)}((x_{(1,\a)}\rightharpoonup_{\a}a_{(1,\g)})_{(2,\g)}\rightharpoonup_\g((x_{(2,\a)}\leftharpoonup_{\g}a_{(2,\g)})\rightharpoonup_\a b))\\
&=(x_{(1,\a)}\rightharpoonup_{\a}a_{(1,\g)})(((x_{(2,\a)}\rightharpoonup_{\a}a_{(2,\g)})\circ(x_{(3,\a)}\leftharpoonup_{\a}a_{(3,\g)}))\rightharpoonup_{\g\a} b)\\
&=(x_{(1,\a)}\rightharpoonup_{\a}a_{(1,\g)})((x_{(2,\a)}\rightharpoonup_{\g}a_{(2,\g)})\circ T_\g(x_{(3,\a)}\rightharpoonup_\g a_{(3,\g)})\circ x_{(4,\a)}\circ a_{(4,\g)})\rightharpoonup_{\g\a} b)\\
&=(x_{(1,\a)}\rightharpoonup_{\a}a_{(1,\g)})( (x_{(2,\a)}\circ a_{(2,\g)})\rightharpoonup_{\a\g} b)\\
&=(x_{(1,\a)}\rightharpoonup_{\a}a_{(1,\g)})(x_{(2,\a)}\rightharpoonup_\a (a_{(2,\g)}\rightharpoonup_\g b))\\
&=x\rightharpoonup_{\a}( a_{(1,\g)}(a_{(2,\g)}\rightharpoonup_\g b))\\
&=x\rightharpoonup_{\a}(a\circ b),
\end{align*}
and
\begin{align*}
&(a\leftharpoonup_\g(b_{(1,\d)}\rightharpoonup_\d x_{(1,\a)}))\circ(b_{(2,\d)}\leftharpoonup_\a x_{(2,\a)}))\\
&=T_\a(a_{(1,\g)}\circ b_{(1,\d)}\rightharpoonup_{\g\d} x_{(1,\a)})\circ a_{(2,\b)}\circ(b_{(2,\d)}\rightharpoonup_\d x_{(2,\a)})\circ T_\a(b_{(3,\d)}\rightharpoonup_\d x_{(3,\a)})\circ b_{(4,\d)}\circ x_{(4,\a)}\\
&=T_\a(a_{(1,\g)}\circ b_{(1,\d)}\rightharpoonup_{\g\d} x_{(1,\a)})\circ a_{(2,\b)}\circ b_{(2,\d)}\circ x_{(2,\a)}\\
&=(a\circ b)\leftharpoonup_\a x.
\end{align*}
\end{proof}

\begin{proposition}\label{pro:mp}
Let $\pi$ be an abelian group and  $(H,\circ,\Delta,T)$ be a cocommutative Hopf $\pi$-algebra. Assume that $(H,H,\rightharpoonup,\leftharpoonup)$ is a matched pair and that 
\begin{equation}\label{mp6}
a\circ b=(a_{(1,\a)}\rightharpoonup_\a b_{(1,\b)})\circ(a_{(2,\a)}\leftharpoonup_\b b_{(2,\b)}),
\end{equation}
for all $a\in H_\a,b\in H_\b$. Then $(H,\cdot,\circ)$ is a Hopf $\pi$-brace with
$$ab=a_{(1,\a)}\circ(T_\a(a_{(2,\a)})\rightharpoonup_{\bar{\a}}b),\ S_\a(a)=a_{(1,\a)}\rightharpoonup_\a T_\a(a_{(2,\a)}).$$
\end{proposition}

\begin{proof}
For all $a\in H_\a,b\in H_\b,c\in H_\g$,
\begin{align*}
\Delta_{\a\b}(ab)&=\Delta_{\a\b}(a_{(1,\a)}\circ(T_\a(a_{(2,\a)})\rightharpoonup_{\bar{\a}}b))\\
&=\Delta_\a(a)\circ\Delta_\b(T_\a(a_{(2,\a)})\rightharpoonup_{\bar{\a}}b)\\
&=\Delta_\a(a)\Delta_\b(b),
\end{align*}
and
\begin{align*}
a_{(1,\a)}(a_{(2,\a)}\rightharpoonup_{\a}b)&=a_{(1,\a)}\circ(T_\a(a_{(2,\a)})\rightharpoonup_{\bar{\a}}(a_{(3,\a)}\rightharpoonup_{\a}b))\\
&=a_{(1,\a)}\circ((T_\a(a_{(2,\a)})\circ a_{(3,\a)})\rightharpoonup_{e}b))\\
&=a_{(1,\a)}\circ(1\circ b)\\
&=a\circ b.
\end{align*}
Since
\begin{align*}
&a\rightharpoonup_\a(bc)=a\rightharpoonup_\a(b_{(1,\b)}\circ(T_\b(b_{(2,\b)})\rightharpoonup_{\bar{\b}}c))\\
&=(a_{(1,\a)}\rightharpoonup_\a b_{(1,\b)})\circ((a_{(2,\a)}\leftharpoonup_\b b_{(2,\b)})\rightharpoonup_\a(T_\b(b_{(2,\b)})\rightharpoonup_{\bar{\b}}c))\\
&=(a_{(1,\a)}\rightharpoonup_\a b_{(1,\b)})((a_{(2,\a)}\rightharpoonup_\a b_{(2,\b)})\rightharpoonup_\b((a_{(3,\a)}\leftharpoonup_\b b_{(3,\b)})\rightharpoonup_\a(T_\b(b_{(4,\b)})\rightharpoonup_{\bar{\b}}c)))\\
&=(a_{(1,\a)}\rightharpoonup_\a b_{(1,\b)})((a_{(2,\a)}\rightharpoonup_\a b_{(2,\b)})\circ(a_{(3,\a)}\leftharpoonup_\b b_{(3,\b)})\rightharpoonup_{\a\b}(T_\b(b_{(4,\b)})\rightharpoonup_{\bar{\b}}c)))\\
&=(a_{(1,\a)}\rightharpoonup_\a b_{(1,\b)})((a_{(2,\a)}\circ b_{(2,\b)})\rightharpoonup_{\a\b}(T_\b(b_{(4,\b)})\rightharpoonup_{\bar{\b}}c))\\
&=(a_{(1,\a)}\rightharpoonup_\a b_{(1,\b)})(a_{(2,\a)} \rightharpoonup_{\a}(b_{(2,\b)}T_\b(b_{(4,\b)})\rightharpoonup_{e}c))\\
&=(a_{(1,\a)}\rightharpoonup_\a b)(a_{(2,\a)}\rightharpoonup_\a c).
\end{align*}
we have
\begin{align*}
a(bc)&=a_{(1,\a)}\circ(T_\a(a_{(2,\a)})\rightharpoonup_{\bar{\a}}(bc))\\
&=a_{(1,\a)}\circ((T_\a(a_{(2,\a)})\rightharpoonup_{\bar{\a}}b_{(1,\b)})(T_\a(a_{(3,\a)})\rightharpoonup_{\bar{\a}}c))\\
&=a_{(1,\a)}\circ(T_\a(a_{(2,\a)})\rightharpoonup_{\bar{\a}}b_{(1,\b)})\circ (T_\b(T_\a(a_{(3,\a)})\rightharpoonup_{\bar{\a}}b_{(2,\b)})\rightharpoonup_{\bar{\b}}(T_\a(a_{(4,\a)})\rightharpoonup_{\bar{\a}}c))\\
&=(a_{(1,\a)}b_{(1,\b)})\circ(T_\b(T_\a(a_{(3,\a)})\rightharpoonup_{\bar{\a}}b_{(2,\b)})\circ T_\a(a_{(4,\a)})\rightharpoonup_{\bar{\a}\bar{\b}}c)\\
&=(a_{(1,\a)}b_{(1,\b)})\circ(T_{\a\b}(a_{(4,\a)}\circ (T_\a(a_{(3,\a)})\rightharpoonup_{\bar{\a}}b_{(2,\b)})) \rightharpoonup_{\bar{\a}\bar{\b}}c)\\
&=(a_{(1,\a)}b_{(1,\b)})\circ(T_{\a\b}(a_{(3,\a)}b_{(2,\b)}) \rightharpoonup_{\bar{\a}\bar{\b}}c)\\
&=(ab)c.
\end{align*}
For the antipode $S$,
\begin{align*}
a_{(1,\a)}S_\a(a_{(2,\a)})&=a_{(1,\a)}(a_{(2,\a)}\rightharpoonup_\a T_\a(a_{(3,\a)}))\\
&=a_{(1,\a)}\circ(T_\a(a_{(2,\a)})\rightharpoonup_{\bar{\a}}(a_{(3,\a)}\rightharpoonup_\a T_\a(a_{(4,\a)})))\\
&=a_{(1,\a)}\circ((T_\a(a_{(2,\a)})\circ a_{(3,\a)})\rightharpoonup_\e T_\a(a_{(4,\a)}))\\
&=a_{(1,\a)}\circ T_\a(a_{(2,\a)})\\
&=\v_\a(a)1,
\end{align*}
and using (\ref{mp6}) we have
$$T_\b(a\rightharpoonup_\a b)=(a_{(1,\a)}\leftharpoonup_\b b_{(1,\b)})\circ T_{\a\b}(a_{(2,\a)}\circ b_{(2,\b)}).$$
Thus
\begin{align*}
&S_\a(a_{(1,\a)})a_{(2,\a)}\\
&=(a_{(1,\a)}\rightharpoonup_\a T_\a(a_{(2,\a)}))\circ(T_{\bar{\a}}(a_{(3,\a)}\rightharpoonup_\a T_\a(a_{(4,\a)}))\rightharpoonup_\a a_{(5,\a)})\\
&=(a_{(1,\a)}\rightharpoonup_\a T_\a(a_{(2,\a)}))\circ((a_{(3,\a)}\leftharpoonup_{\bar{\a}}T_\a(a_{(4,\a)}))\circ T_{e}(a_{(5,\a)}\circ T_\a(a_{(6,\a)}))\rightharpoonup_\a a_{(7,\a)})\\
&=(a_{(1,\a)}\rightharpoonup_\a T_\a(a_{(2,\a)}))\circ((a_{(3,\a)}\leftharpoonup_{\bar{\a}}T_\a(a_{(4,\a)}))\rightharpoonup_\a a_{(5,\a)})\\
&=a_{(1,\a)}\rightharpoonup_\a( T_\a(a_{(2,\a)})a_{(3,\a)})\\
&=\v_\a(a)1.
\end{align*}
For the compatibility condition, we have
\begin{align*}
&(a_{(1,\a)}\circ b)S(a_{(2,\a)})(a_{(3,\a)}\circ c)\\
&=(a_{(1,\a)}\circ b)S(a_{(2,\a)})a_{(3,\a)}(a_{(4,\a)}\rightharpoonup_\a c)\\
&=(a_{(1,\a)}\circ b)(a_{(2,\a)}\rightharpoonup_\a c)\\
&=(a_{(1,\a)}\circ b_{(1,\b)})\circ(T_{\a\b}(a_{(2,\a)}\circ b_{(2,\b)})\rightharpoonup_{\bar{\a}\bar{\b}}(a_{(3,\a)}\rightharpoonup_\a c))\\
&=(a_{(1,\a)}\circ b_{(1,\b)})\circ((T_{\b}(b_{(2,\b)})\circ T_{\a}(a_{(2,\a)})\circ a_{(3,\a)})\rightharpoonup_{\bar{\b}} c)\\
&=a\circ (b_{(1,\b)}\circ(T_{\b}(b_{(2,\b)})\rightharpoonup_{\bar{\b}} c))\\
&=a\circ (b_{(1,\b)}\circ(bc).
\end{align*}
The proof is completed.
\end{proof}

Let $(H,\cdot,1,\Delta,\v,S)$ be a cocommutative Hopf $\pi$-algebra and $\mathbf{M_p}(H)$ the category consisting of objects the matched pair $(H,H)$ satisfying the relation (\ref{mp6}). The morphism in $\mathbf{M_p}(H)$ is a morphism of Hop $\pi$-algebra $f:H\rightarrow H$ such that for all $a\in H_\a,b\in H_\b$, $f(a\rightharpoonup_\a b)=f(a)\rightharpoonup_\a f(b),\ f(a\leftharpoonup_\b b)=f(a)\leftharpoonup_\b f(b)$.

\begin{theorem}
Let $\pi$ be an abelian group and $(H,\cdot,1,\Delta,\v,S)$ a cocommutative Hopf $\pi$-algebra. The category $\mathbf{Br}(H)$ and $\mathbf{M_p}(H)$ are equivalent.
\end{theorem}

\begin{proof}
Define $F:\mathbf{Br}(H)\rightarrow \mathbf{M_p}(H)$ by $F((H,\cdot,\circ))=(H_\circ,H_\circ)$, where the matched pair $(H_\circ,H_\circ)$ is given in Proposition \ref{pro:brace}. For any morphism in $\mathbf{Br}(H)$, $F(f)=f$. Clearly $F$ is a functor.

Conversely define $G:\mathbf{M_p}(H)\rightarrow\mathbf{Br}(H)$ by $F(H,H)=(H,\cdot,\circ)$, where $(H,\cdot,\circ)$ is the Hopf brace given in Proposition \ref{pro:mp}, and $G(f)=f$ for any morphism in $\mathbf{M_p}(H)$. It is easy to see that $G$ is a functor. By an routine exercise, $\mathbf{Br}(H)$ and $\mathbf{M_p}(H)$ are equivalent.
\end{proof}

\section{Post-Hopf group algebras}
 \def\theequation{3.\arabic{equation}}
 \setcounter{equation} {0}

\begin{definition}   \label{defi:pH}
A {\bf post-Hopf $\pi$-algebra} is a pair $(H,\tr)$, where $H=(\{H_\a\}_{\a\in\pi},\cdot,1,\D,\v,S)$ is a Hopf $\pi$-algebra and $\tr=\{\tr_{\a,\b}:H_\a\o H_\b\rightarrow H_\b\}$ is a family of coalgebra homomorphisms satisfying 
\begin{eqnarray}
\label{eq:P1} &x\tr_{\a,\b\g}(yz)=(x_{(1,\a)}\tr_{\a,\b}y)(x_{(2,\a)}\tr_{\a,\g}z),&\\
\label{eq:P2} &x\tr_{\a,\g}(y\tr_{\b,\g}z)=(x_{(1,\a)}(x_{(2,\a)}\tr_{\a,\b}y))\tr_{\a\b,\g}z,&
\end{eqnarray}
for any $x\in H_\a,y\in H_\b,z\in H_\g$, and the linear map $\theta_{\a,\b}:H_\a\rightarrow \End(H_\b)$ defined by
$$\theta_{\a,\b,x}(y)=x\tr_{\a,\b}y$$
is convolution invertible in $\Hom(H_\a,\End(H_\b))$, that is, there exists unique $\psi_{\a,\b}\in\Hom(H_\a,\End(H_\b))$ such that 
\begin{equation}\label{eq:P-con}
  \psi_{\a,\b,x_{(1,\a)}}\theta_{\a,\b,x_{(2,\a)}}=\theta_{\a,\b,x_{(1,\a)}}\psi_{\a,\b,x_{(2,{\a,\b})}}=\v_\a(x)id_{H_\b}.
\end{equation}
\end{definition}

\begin{remark}
(1) When $\pi=\{1\}$, a post-Hopf $\pi$-algebra could be reduced to a post-Hopf algebra introduced in \cite{LST}.

(2) A post-Hopf $\pi$-algebra $(H,\tr)$ is called cocommutative if each $H_\a$ is cocommutative.
\end{remark}

\begin{lemma}
Let $(H,\tr)$ be a post-Hopf $\pi$-algebra. Then for all $x\in H_\a,y\in H_\b, \a,\b\in\pi$, we have
\begin{eqnarray}
\label{eq:P3} x\tr_{\a,e} 1 &=& \v_\a(x)1, \\
\label{eq:P4} 1\tr_{e,\a} x &=& x, \\
\label{eq:P5} S_\b(x\tr_{\a,\b} y) &=& x\tr_{\a,\bar{\b}} S_\b(y).
\end{eqnarray}
\end{lemma}
\begin{proof}
Since $\tr$ is a coalgebra homomorphism, we have
\begin{eqnarray*}
x\tr_{\a,e} 1 &=& (x_{(1,\a)}\tr_{\a,e} 1) \v(x_{(2,\a)}\tr_{\a,e} 1)\\
 &=& (x_{(1,\a)}\tr_{\a,e} 1)\cdot (x_{(2,\a)}\tr_{\a,e} 1)\cdot S_e(x_{(3,\a)}\tr_{\a,e} 1)\\
 &\stackrel{\ref{eq:P1}}{=}& (x_{(1,\a)}\tr_{\a,e} 1)\cdot S(x_{(2,\a)}\tr_{\a,e} 1)\\
 &=& \v_e(x\tr_{\a,e} 1)1 =\v_e(x)1.
\end{eqnarray*}

By Eq.~\eqref{eq:P-con},  we have
$\psi_{e,\a,1}\theta_{e,\a,1}=\theta_{e,\a,1}\psi_{e,\a,1}=id_{H_\a},$
which means that $\theta_{e,\a,1}$ is a linear automorphism of $H_\a$. On the other hand, we have
\begin{eqnarray*}
\theta_{e,\a,1}^2(x) = 1\tr_{e,\a} (1\tr_{e,\a} x)
\stackrel{\eqref{eq:P2}}{=}  (1\tr_{e,\a} 1)\tr_{e,\a} (x)
\stackrel{\eqref{eq:P3}}{=}  1\tr_{e,\a} x = \theta_{e,\a,1}(x).
\end{eqnarray*}
Hence, $1\tr_{e,\a} (x) = \theta_{e,\a,1}(x) = x$.
Finally we have
\begin{eqnarray*}
S_\b(x\tr_{\a,\b} y)&=&S_\b(x_{(1,\a)}\tr_{\a,\b} y_{(1,\b)})\v_\a(x_{(2,\a)})\v_\b(y_{(2,\b)}) \\
 &\stackrel{\ref{eq:P3}}{=}& S_\b(x_{(1,\a)}\tr_{\a,\b} y_{(1,\b)})\cdot(x_{(2,\a)}\tr_{\a,\b} \v_\b(y_{(2,\b)})1)\\
&=&S_\b(x_{(1,\a)}\tr_{\a,\b} y_{(1,\b)})\cdot(x_{(2,\a)}\tr_{\a,\b} (y_{(2,\b)}\cdot S_\b(y_3))) \\
 &\stackrel{\ref{eq:P1}}{=}& S_\b(x_{(1,\a)}\tr_{\a,\b} y_{(1,\b)})\cdot(x_{(2,\a)}\tr_{\a,\b} y_{(2,\b)})\cdot(x_{(3,\a)}\tr_{\a,\bar{\a}} S_\b(y_{(3,\b)}))\\
&=&\v_\b(x_{(1,\a)}\tr_{\a,\b} y_{(1,\b)})(x_{(2,\a)}\tr_{\a,\b} S_\b(y_{(2,\b)}))\\
&=&\v_\a(x_{(1,\a)})\v_\b(y_{(1,\b)})(x_{(2,\a)}\tr S_\b(y_{(2,\b)}))=x\tr S_\b(y). 
\end{eqnarray*}
These finish the proof.
\end{proof}

\begin{theorem}  \label{thm:subHopf}
Let $(H,\tr)$ be a cocommutative post-Hopf $\pi$-algebra. Define
\begin{eqnarray}
\label{post-rbb-1} x *_{\a,\b} y &:=& x_{(1,\a)}\cdot (x_{(2,\b)}\tr_{\a,\b} y),\\
\label{post-rbb-2} T_\a(x) &:=& \psi_{\a,\bar{\a},x_{(1,\a)}}(S_\a(x_{(2,\a)})),
\end{eqnarray}
for all $x\in H_\a,y\in H_\b, \a,\b\in\pi$. Then $H_\tr:=(H,*,1,\Delta,\v,T)$ is a Hopf $\pi$-algebra, which is called the {\bf subadjacent Hopf $\pi$-algebra}. Furthermore $(H_e,\cdot,1,\Delta_e,\v_e,S_e)$ is a left $H_\tr$-modulelike $\pi$-bialgebra under the action $\{\tr_{\a,e}\}_{\a\in\pi}$.
\end{theorem}

\begin{proof}
Since $\tr$ is a coalgebra homomorphism and $H$ is cocommutative, we have
\begin{eqnarray*}
\Delta_{\a,\b}(x *_{\a,\b} y)&=&\Delta_{\a,\b}(x_{(1,\a)}\cdot (x_{(2,\b)}\tr_{\a,\b} y))\\
&=&\Delta_\a(x_{(1,\a)})\cdot \Delta_{\a,\b}(x_{(2,\a)}\tr_{\a,\b} y)\\
&=&(x_{(1,\a)}\otimes x_{(2,\a)})\cdot((x_{(3,\a)}\tr_{\a,\b} y_{(1,\b)})\otimes(x_{(4,\a)}\tr_{\a,\b} y_{(2,\b)}))\\
&=&(x_{(1,\a)}\cdot (x_{(3,\a)}\tr_{\a,\b} y_{(1,\b)}))\otimes (x_{(2,\a)}\cdot(x_{(4,\a)}\tr_{\a,\b} y_{(2,\b)}))\\
&=&(x_{(1,\a)}\cdot (x_{(2,\a)}\tr_{\a,\b} y_{(1,\b)}))\otimes (x_{(3,\a)}\cdot(x_{(4,\a)}\tr_{\a,\b} y_{(2,\b)}))\\
&=&(x_{(1,\a)} *_{\a,\b} y_{(1,\b)})\otimes(x_{(2,\a)} *_{\a,\b} y_{(2,\b)})
\end{eqnarray*}
for all $x\in H_\a,y\in H_\b, \a,\b\in\pi$, which implies  that the comultiplication $\Delta$ is an algebra homomorphism with respect to the multiplication $*_{\a,\b}$. 
Moreover, we have
\begin{align*}
  \v_\b(x *_{\a,\b} y) & =\v_\b(x_{(1,\a)}\cdot (x_{(2,\a)}\tr_{\a,\b} y)) \\
  & =\v_\a(x_{(1,\a)})\v_\b(x_{(2,\a)}\tr_{\a,\b} y)=\v_\a(x)\v_\b(y),
\end{align*}
which implies  that the counit $\v$ is also an algebra homomorphism with respect to the multiplication $*_{\a,\b}$. 
Since the comultiplication $\Delta$ is an algebra homomorphism with respect to the multiplication $\cdot$, 
for all $x\in H_\a,y\in H_\b,z\in H_\g, \a,\b,\g\in\pi$, we have
\begin{eqnarray*}
(x *_{\a,\b} y)*_{\a\b,\g} z&=& (x_{(1,\a)} *_{\a,\b} y_{(1,\b)})\cdot((x_{(2,\a)} *_{\a,\b} y_{(2,\b)})\tr_{\a\b,\g} z)\\
&=&(x_{(1,\a)}\cdot (x_{(2,\a)}\tr_{\a,\b} y_{(1,\b)}))\cdot((x_{(3,\a)}\cdot (x_{(4,\a)}\tr_{\a,\b} y_{(2,\b)}))\tr_{\a\b,\g} z)\\
&\stackrel{\eqref{eq:P2}}{=}&x_{(1,\a)}\cdot (x_{(2,\a)}\tr_{\a,\b} y_{(1,\b)})\cdot (x_{(3,\a)}\tr_{\a,\g} (y_{(2,\b)}\tr_{\b,\g} z))\\
&\stackrel{\eqref{eq:P1}}{=}&x_{(1,\a)}\cdot (x_{(2,\a)}\tr_{\a,\b} (y_{(1,\b)}\cdot (y_{(2,\b)}\tr_{\b,\g} z)))\\
&=&x *_{\a,\b\g} (y*_{\b,\g} z),
\end{eqnarray*}
which implies that the multiplication $*_{\a,\b}$ is associative.  
For any $x\in H_\a, \a\in\pi$, by  \eqref{eq:P3} and \eqref{eq:P4}, we have
\begin{eqnarray*}
x *_{\a,e} 1&=&x_{(1,\a)}\cdot (x_{(2,\a)}\tr_{\a,\b} 1)=x_{(1,\a)}\cdot\v_\a(x_{(2,\a)})=x,\\
1 *_{e,\a} x&=&1\cdot (1\tr_{\a,\b} x)=x.
\end{eqnarray*}
Since $\tr$ is a coalgebra homomorphism and $H$ is cocommutative, we know that
$$\Delta_\b\psi_{\a,\b,x}=(\psi_{\a,\b,x_{(1,\a)}}\otimes\psi_{\a,\b,x_{(2,\a)}})\Delta_\b,$$
and $T$ is a coalgebra homomorphism. Also, note that
\begin{eqnarray*}
x_{(1,\a)}*_{\a,\bar{\a}} T_\a(x_{(2,\a)}) 
&\stackrel{\ref{post-rbb-1}}{=}& x_{(1,\a)}\cdot(x_{(2,\a)}\tr_{\a,\bar{\a}} T_\a(x_{(2,\a)}))\\
&\stackrel{\ref{post-rbb-2}}{=}& x_{(1,\a)}\cdot(\theta_{\a,\bar{\a},x_{(2,\a)}}(\psi_{\a,\bar{\a},x_{(3,\a)}}(S_\a(x_{(4,\a)}))))\\
&=& x_{(1,\a)}\cdot(\v_\a(x_{(2,\a)})S_\a(x_{(3,\a)}))\\
&=& \v_\a(x)1,
\end{eqnarray*}
and this implies that
\begin{eqnarray*}
T_{\bar{\a}}T_\a(x) &=& \v_\a(x_{(1,\a)})T_{\bar{\a}}T_\a(x_{(2,\a)}) \\
&=& (x_{(1,\a)}*_{\a,\bar{\a}} T_{\a}(x_{(2,\a)}))*_{e,\a} T_{\bar{\a}}T_\a(x_{(3,\a)})\\
&=& x_{(1,\a)}*_{\a,e} (T_{\a}(x_{(2,\a)})*_{\bar{\a},\a} T_{\bar{\a}}(T_\a(x_{(3,\a)})))\\
&=& x_{(1,\a)}*_{\a,e} \v_{\bar{\a}}(T_\a(x_{(2,\a)}))1 = x,
\end{eqnarray*}
and
\begin{eqnarray*}
T_\a(x_{(1,\a)}) *_{\bar{\a},\a} x_{(2,\a)} &=&  T_\a(x_{(1,\a)}) *_{\bar{\a},\a} T_{\bar{\a}}T_\a(x_{(2,\a)})\\
 &=& \v_{\bar{\a}}(T_\a(x))1 = \v_\a(x)1.
\end{eqnarray*}
Therefore, $(H,*,1,\Delta,\v,T)$ is a cocommutative Hopf $\pi$-algebra.

Moreover, we have
$$ (x *_{\a,\b} y)\tr_{\a\b,\g} z=(x_{(1,\a)}\cdot (x_{(2,\a)}\tr_{\a,\b} y))\tr_{\a\b,\g} z=x\tr_{\a,\g} (y\tr_{\b,\g} z). $$
Then by (\ref{eq:P1}) and (\ref{eq:P3}), $(H_e ,\cdot,1)$ is a left $\pi$-$H_\tr$-modulelike. 
Since $\tr_{(\a,\b)}$ is also a coalgebra homomorphism, $(H_e,\cdot,1,\Delta_e,\v_e,S_e)$ is a left $\pi$-$H_\tr$-modulelike $\pi$-bialgebra via the action $\tr_{\a,\b}$.
\end{proof}

\begin{corollary}
Let $(H,\tr)$ be a cocommutative post-Hopf $\pi$-algebra.
We can construct a Hopf $\pi$-brace $H_e\natural H_\tr=\{H_e\sharp H_\a\}_{\a\in\pi}$ via Theorem \ref{thm:subHopf} and Example \ref{ex:modlike}.
\end{corollary}

\begin{theorem}  \label{thm:br}
Let $(H,\tr)$ be a cocommutative post-Hopf $\pi$-algebra. Then $(H,\cdot,1,\Delta,\v,S)$  and  the subadjacent Hopf $\pi$-algebra $(H,*,1,\Delta,\v,T)$ form a Hopf $\pi$-brace. 
Conversely, any cocommutative Hopf $\pi$-brace $(H,\cdot,\circ)$ gives a post-Hopf $\pi$-algebra  $(H,\tr)$ with $\tr$ defined by 
$$x\tr_{\a,\b} y=S_\a(x_{(1,\a)})\cdot_{\bar{\a},\b}(x_{(2,\a)}\circ_{\a,\b} y), \quad \forall x\in H_\a,y\in H_\b, \a,\b\in\pi.$$
\end{theorem}
\begin{proof}
Let $(H,\tr)$ be a cocommutative post-Hopf $\pi$-algebra. We only need to show that the multiplications $\cdot$ and $*$ satisfy the compatibility  condition (\ref{eq:Hbrace}), which follows from
\begin{align*}
x*_{\a,\b}(y\cdot z)&=x_{(1,\a)}\cdot (x_{(2,\a)}\tr_{\a,\b\g}(y\cdot z))\\
&=x_{(1,\a)}\cdot((x_{(2,\a)}\tr_{\a,\b} y)\cdot (x_{(3,\a)}\tr_{\a,\g} z))\\
&=x_{(1,\a)}\cdot(x_{(2,\a)}\tr_{\a,\b} y)\cdot S_\a(x_{(3,\a)})\cdot x_{(4,\a)}\cdot (x_5\tr_{\a,\g} z)\\
&=(x_{(1,\a)}*_{\a,\b} y)\cdot S_\a(x_{(2,\a)})\cdot (x_{(3,\a)}*_{\a,\b} z),
\end{align*}
for any $x\in H_\a,y\in H_\b,z\in H_\g, \a,\b,\g\in\pi$.

Conversely, it is straightforward but tedious to check that a cocommutative Hopf $\pi$-brace $(H,\cdot,\circ)$ induces a post-Hopf $\pi$-algebra  $(H,\tr)$.
\end{proof}

\section{Rota-Baxter operators on cocommutative Hopf $\pi$-algebras}
 \def\theequation{4.\arabic{equation}}
 \setcounter{equation} {0}

\begin{definition}
Let $H=\{H_\a\}_{\a\in\pi}$ be a cocommutative Hopf $\pi$-algebra. The family of coalgebra homomorphisms $B=\{B_\a:H_\a\rightarrow H_{\bar{\a}}\}_{\a\in\pi}$ is called a Rota-Baxter operator on $H$ if for all $\a,\b\in\pi, a\in H_\a,b\in H_\b$,
\begin{equation} \label{eq:RBO}
  B_\a(a)B_\b(b)=B_{\b\a}\big(a_{(1,\a)}B_\a(a_{(2,\a)})bS_{\bar{\a}}(B_\a(a_{(3,\a)}))\big).
\end{equation}
The pair $(H,B)$ is called a Rota-Baxter Hopf $\pi$-algebra.
\end{definition}

\begin{remark}
\begin{enumerate}[\rm(1)]
  \item When the group $\pi=\{1\}$, we could recover the notion of Rota-Baxter Hopf algebras.
  \item It is obvious that the antipode $S$ is a Rota-Baxter operator on $H$.
  \item Let $(H,B)$ be a Rota-Baxter Hopf $\pi$-algebra, and $\varphi$ a $\pi$-bialgebra automorphism or antiautomorphism of $H$. Then $B^{(\varphi)}=\{B^{(\varphi)}_\a=\varphi_{\bar{\a}}\circ B_\a\circ\varphi^{-1}_\a\}_{\a\in\pi}$ is also a  Rota-Baxter operator on $H$.
\end{enumerate}
\end{remark}

\begin{example}
Let $(H,\Delta,\v,S,B)$ be a Rota-Baxter Hopf algebra and $\pi=Aut(H)$ the group of Hopf algebras automorphisms of $H$. For each $\a\in\pi$, $H_\a=H$ as a vector space. We denote the element in $H_\a$ by $h^\a=\a(h)$ for $h\in H$. Define
\begin{align*}
&m_{\a,\b}:H_\a\o H_\b\rightarrow H_{\a\b},\ g^\a\o h^\b\mapsto (gh)^{\a\b},\\
&\Delta_\a=\Delta,\quad \varepsilon_\a=\varepsilon,\\
&S_\a:H_\a\rightarrow H_{\bar{\a}},\ h^\a\mapsto S(h)^{\bar{\a}},\\
&B_\a:H_\a\rightarrow H_{\bar{\a}},\ h^\a\mapsto B(h)^{\bar{\a}},
\end{align*}
then $(\{H_\a\}_{\a\in \pi},\{B_\a\}_{\a\in \pi})$ is a Rota-Baxter Hopf $\pi$-algebra.
\end{example}

\begin{theorem}
Let $H=\{H_\a\}_{\a\in\pi}$ be a cocommutative Hopf $\pi$-algebra. Suppose that $G_\a$ is a Hopf subalgebra of $H_e$ and $K_\a$ is a subcoalgebra of $H_\a$ for all $\a\in\pi$, such that $G=\{G_\a\}_{\a\in\pi}$ is a Hopf $\pi$-algebra and $K=\{K_\a\}_{\a\in\pi}$ is a Hopf $\pi$-subalgebra of $H_\a$, and as a Hopf $\pi$-algebra $H=GK$, i.e., $H_\a=G_\a K_\a$ for all $\a\in\pi$. 
Suppose that the product is direct, that is, $H$ is isomorphic to $G\otimes K=\{G_\a \otimes K_\a\}_{\a\in\pi}$ as a family of vector spaces. Define a family of maps $B=\{B_\a:H_\a\rightarrow H_{\bar{\a}}\}_{\a\in\pi}$ by
$$B_\a (h h')=\v_e(h)S_\a(h'),$$
where $h_1\in G_\a, h_2\in K_\a$. 
Then $B$ is a Rota-Baxter operator of $H$.
\end{theorem}
\begin{proof}
Clearly, $B$ is a family of well-defined linear maps. 
First we prove that each $B_\a$ is a coalgebra map. For $x=hg\in H_\a$, where $h\in G_\a$, $g\in K_\a$ we have:
\begin{eqnarray*}
\Delta_{\bar{\a}}(B_\a(x))&=&\v_e(h)\Delta_{\bar{\a}}(S_\a(g))=\v_e(h) S_{\bar{\a}}(g_{(2,\a)})\otimes S_{\bar{\a}}(g_{(1,\a)})\\
&=&\v_e(h_{(1,e)})S_{\bar{\a}}(g_{(1,\a)})\otimes \v_e(h_{(2,e)})S_{\bar{\a}}(g_{(2,\a)})\\
&=&(B_{\bar{\a}}\otimes B_{\bar{\a}})\Delta_\a(x).
\end{eqnarray*}

In order to prove that $B$ satisfies (\ref{eq:RBO}) consider $x=hg\in H_\a$  and $y=h'g'\in H_\b$, where $h\in G_\a, g\in K_\a$, $h'\in G_\b, g'\in K_\b$ for all $\a,\b\in\pi$. We have
\begin{eqnarray*}
 &&B_{\b\a}(x_{(1,\a)}B_\a(x_{(2,\a)})yS_{\bar{\a}}(B_\a(x_{(3,\a)})))\\
 &=&B_{\b\a}((h_{(1,e)}g_{(1,\a)})(\v_e(h_{(2,e)})S_\a (g_{(2,\a)}))(h'g')(\v_e(h_{(3,e)})S_{\bar{\a}}(S_\a(g_{(3,\a)}))))\\
 &=&B_{\b\a}(\v_e(h_{(1,e)})\v_e(h_{(2,e)})h_{(3,e)}g_{(1,\a)}S_\a(g_{(2,\a)}) h'g' g_{(3,\a)})\\
 &=&B_{\b\a}(hh' g' \v_e(g_{(1,\a)})g_{(2,\a)}) = B_{\b\a}(hh'g'g)\\
 &=&\v_{e}(hh')S_{\b\a}(g'g)=\v_e(h)\v_e(h')S_\a(g)S_\b(g')\\
 &=&B_\a(x)B_\b(y).
\end{eqnarray*}

Since $H_\a=G_\a K_\a$ is spanned by elements of the form $hg$, where $h\in G_\a$, $g\in K_\a$ for all $\a\in\pi$, the equation (\ref{eq:RBO}) holds for all $x,y\in H$.
\end{proof}

Let $\pi$ be an abelian group, then $H_{\a\b}=H_{\b\a}$ for a Hopf $\pi$-algebra $H$.
We now construct a Hopf $\pi$-brace via Rota-Baxter Hopf $\pi$-algebras by Theorem \ref{thm:hb}.

\begin{lemma}\label{lem:nH}
Let $(H=\{H_\a\}_{\a\in\pi},B)$ be a Rota-Baxter Hopf $\pi$-algebra. Define
\begin{eqnarray}
& m_{\a,\b}(g,h)=g\circ_B h:=g_{(1,\a)}B_\a(g_{(2,\a)})hS_{\bar{\a}}(B_\a(g_{(3,\a)})), & \label{eq:RBHm}\\
& T_\a(g):=S_{\bar{\a}}(B_\a(g_{(1,\a)}))S_\a(g_{(2,\a)})B_\a(g_{(3,\a)}), & \label{eq:RBHS}
\end{eqnarray}
for all $g\in H_\a,h\in H_\b, ~\a,\b\in\pi$.
\begin{enumerate}[\quad\rm(1)]
\item $H_B:=(H,m=\{m_{\a,\b}\}_{\a\in\pi},T=\{T_{\a,\b}\}_{\a\in\pi})$ is a cocommutative Hopf $\pi$-algebra. 
\item For all $h\in H_\a$, $\a\in\pi$, we have
\begin{eqnarray}
 & B_\a(h_{(1,\a)})B_{\bar{\a}}(T_\a(h_{(2,\a)}))=\v_\a(h)1_e, \label{eq:t} & \\
 & B_{\bar{\a}}T_\a = S_{\bar{\a}}B_\a. &
\end{eqnarray}

\item The operator $B$ is also a Rota-Baxter operator on $H_B$, that is, $(H_B,B)$ is a Rota-Baxter Hopf $\pi$-algebra.
\item The family of maps $B:H_B\rightarrow H$ is a homomorphism of Rota-Baxter Hopf $\pi$-algebras.
\end{enumerate}
\end{lemma}
\begin{proof}
(1) First, we prove $m=\{m_{\a,\b}\}_{\a,\b\in\pi}$ is  associative. For all $g\in H_\a, h\in H_\b, \ell\in H_\g, ~\a,\b,\g\in\pi$,
\begin{eqnarray*}
& & (g\circ_B h) \circ_B\ell \\
&=& (g_{(1,\a)}B_\a(g_{(2,\a)})hS_{\bar{\a}}(B_\a(g_{(3,\a)}))) \circ_B\ell \\
&=& g_{(1,\a)}B_\a(g_{(2,\a)})h_{(1,\b)}S_{\bar{\a}}(B_\a(g_{(3,\a)}))  B_{\b\a}\big( g_{(4,\a)} B_\a(g_{(5,\a)})h_{(2,\b)}S_{\bar{\a}}(B_\a(g_{(6,\a)})) \big) \\
&&\quad \ell S_{\bar{\b\a}}\big(B_{\b\a}( g_{(7,\a)} B_\a(g_{(8,\a)})h_{(3,\b)}S_{\bar{\a}}(B_\a(g_{(9,\a)})) )\big) \\
&=& g_{(1,\a)}B_\a(g_{(2,\a)})h_{(1,\b)}S_{\bar{\a}}(B_\a(g_{(3,\a)}))  B_\a(g_{(4,\a)})   B_\b(h_{(2,\b)})\ell  S_{\bar{\b\a}}( B_\a(g_{(3,\a)}) B_\b(h_{(3,\b)}) ) \\
&=& g_{(1,\a)}B_\a(g_{(2,\a)})h_{(1,\b)} B_\b(h_{(2,\b)})\ell S_{\bar{\b}}(B_\b(h_{(3,\b)})) S_{\bar{\a}}(B_\a(g_{(3,\a)})) \\
&=& g\circ_B( h_{(1,\b)} B_\b(h_{(2,\b)})\ell S_{\bar{\b}}(B_\b(h_{(3,\b)})) ) \\
&=& g\circ_B(h\circ_B \ell).
\end{eqnarray*}

  Then, we prove $m=\{m_{\a,\b}\}_{\a,\b\in\pi}$ is a $\pi$-coalgebra homomorphism. For all $g\in H_\a, h\in H_\b, ~\a,\b\in\pi$,
\begin{eqnarray*}
& & (g\circ_B h)_{(1,\a\b)} \otimes (g\circ_B h)_{(2,\a\b)} = (g\circ_B h)_{(1,\b\a)} \otimes (g\circ_B h)_{(2,\b\a)} \\
&=& (g_{(1,\a)}B_\a(g_{(2,\a)})hS_{\bar{\a}}(B_\a(g_{(3,\a)})))_{(1,\b\a)} \otimes (g_{(1,\a)}B_\a(g_{(2,\a)})hS_{\bar{\a}}(B_\a(g_{(3,\a)})))_{(2,\b\a)}  \\
&=& g_{(1,\a)}B_\a(g_{(2,\a)})h_{(1,\b)}S_{\bar{\a}}(B_\a(g_{(3,\a)})) \otimes g_{(4,\a)}B_\a(g_{(5,\a)})h_{(2,\b)}S_{\bar{\a}}(B_\a(g_{(6,\a)}))  \\
&=& (g_{(1,\a)}\circ_B h_{(1,\b)})\otimes(g_{(2,\a)}\circ_B h_{(2,\b)}).
\end{eqnarray*}

  Last, we prove $T=\{T_\a\}_{\a\in\pi}$ is an antipode. For all $h\in H_\a, ~\a\in\pi$,
\begin{eqnarray*}
& & m_{\a,\bar{\a}}(id_{H_{\a}} \otimes T_\a)\Delta_\a(h) \\
&=& h_{(1,\a)} \circ_B S_{\bar{\a}}(B_\a(h_{(2,\a)}))S_\a(h_{(3,\a)})B_\a(h_{(4,\a)}) \\
&=& h_{(1,\a)}B_\a(h_{(2,\a)}) S_{\bar{\a}}(B_\a(h_{(3,\a)}))S_\a(h_{(4,\a)})B_\a(h_{(5,\a)}) S_{\bar{\a}}(B_\a(h_{(6,\a)})) \\
&=& h_{(1,\a)} \v_\a(h_{(2,\a)}) S_\a(h_{(3,\a)}) \v_\a(h_{(4,\a)}) 1_e \\
&=& \v_\a(h_{(1,\a)}) \v_\a(h_{(2,\a)}) 1_e \\
&=& \v_\a(h)1_e.
\end{eqnarray*}

So $m_{\a,\bar{\a}}(id_{H_{\a}} \otimes T_\a)\Delta_\a = \v_\a 1_e$. Similarly prove $m_{\bar{\a},\a}(T_\a \otimes id_{H_{\a}})\Delta_\a = \v_\a 1_e$. 
Hence, $H$ with the new multiplication $m=\{m_{\a,\b}\}_{\a\in\pi}$ and antipode $T=\{T_{\a,\b}\}_{\a\in\pi})$  is a cocommutative Hopf $\pi$-algebra

(2) For all $h\in H_\a, ~\a\in\pi$, we have
\begin{eqnarray*}
& & B_\a(h_{(1,\a)})B_{\bar{\a}}(T_\a(h_{(2,\a)})) \\
&=& B_\a(h_{(4,\a)}) B_{\bar{\a}}(S_{\bar{\a}}(B_\a(h_{(2,\a)}))S_\a(h_{(3,\a)})B_\a(h_{(4,\a)}))  \\
&=& B_e(h_{(1,\a)}B_\a(h_{(2,\a)}) S_{\bar{\a}}(B_\a(h_{(3,\a)}))S_\a(h_{(4,\a)})B_\a(h_{(5,\a)}) S_{\bar{\a}}(B_\a(h_{(6,\a)}))) \\
&=& B_e(h_{(1,\a)} \v_\a(h_{(2,\a)}) S_\a(h_{(3,\a)}) \v_\a(h_{(4,\a)}) 1_e) \\
&=& \v_\a(h_{(1,\a)}) \v_\a(h_{(2,\a)}) B_e(1_e) \\
&=& \v_\a(h)1_e.
\end{eqnarray*}

From the above and $B_\a(h_{(1,\a)})S_{\bar{\a}}(B_\a(h_{(2,\a)}))=\v_\a(B_\a(h))=\v_\a(x)1_e$, we have $B_{\bar{\a}}T_\a$ and $S_{\bar{\a}}B_\a$ are the convolutional inverse for $B_\a$. Hence $B_{\bar{\a}}T_\a = S_{\bar{\a}}B_\a$, for all $\a\in\pi$.

(3) From Eq. (\ref{eq:RBO}) and (\ref{eq:RBHm}), we have 
\begin{eqnarray*}
 B_{\a\b}(g \circ_B h) &=& B_{\b\a}(g \circ_B h) \\
 &=& B_{\b\a}( g_{(1,\a)}B_\a(g_{(2,\a)})h S_{\bar{\a}}(B_\a(g_{(3,\a)})) ) \\
 &=& B_\a(g) B_\b(h),
\end{eqnarray*}
for all $g\in H_\a, h\in H_\b, ~\a,\b\in\pi$.

Then we have
\begin{eqnarray*}
& & B_{\b\a}\big(g_{(1,\a)}\circ_B B_\a(g_{(2,\a)})\circ_B h \circ_B T_{\bar{\a}}(B_\a(g_{(3,\a)}))\big) \\
&=& B_{\a}(g_{(1,\a)}) B_{\bar{\a}}(B_\a(g_{(2,\a)})) B_{\b}(h) B_{\a}(T_{\bar{\a}}(B_\a(g_{(3,\a)}))) \\
&=& B_{\a}(g_{(1,\a)}) B_{\bar{\a}}(B_\a(g_{(2,\a)})) B_{\b}(h) S_{\a}(B_{\bar{\a}}(B_\a(g_{(3,\a)})))  \\
&=& B_{\a}(g) \circ_B B_{\a}(h).
\end{eqnarray*}
for all $g\in H_\a, h\in H_\b, ~\a,\b\in\pi$. The operator $B$ is also a Rota-Baxter operator on the cocommutative Hopf $\pi$-algebra $H_B$.

(4) From (3), it is evident that the family of maps $B:H_B\rightarrow H$ is a homomorphism of Rota-Baxter Hopf $\pi$-algebras.
\end{proof}

\begin{theorem}  \label{thm:hb}
Let $(H=\{H_\a\}_{\a\in\pi},\cdot,B)$ be a Rota-Baxter Hopf $\pi$-algebra, and  the multiplication $\circ_B$ and antipode $T$ defined as (\ref{eq:RBHm}) and (\ref{eq:RBHS}). 
Then $(H,\cdot,\circ_B)$ is a cocommutative Hopf $\pi$-brace.
\end{theorem}
\begin{proof}
By Lemma \ref{lem:nH}, $(H,\circ_B,T)$ is a cocommutative Hopf $\pi$-algebra. For all $g\in H_\a, h\in H_\b, \ell\in H_\g, ~\a,\b,\g\in\pi$, we have
\begin{eqnarray*}
& & (g_{(1,\a)}\circ_B h)S_{\bar{\a}}(g_{(2,\a)})(g_{(3,\a)}\circ_B \ell) \\
&=& g_{(1,\a)}B_\a(g_{(2,\a)})hS_{\bar{\a}}(B_\a(g_{(3,\a)})) S_{\bar{\a}}(g_{(4,\a)}) g_{(5,\a)}B_\a(g_{(6,\a)})\ell S_{\bar{\a}}(B_\a(g_{(7,\a)})) \\
&=& g_{(1,\a)}B_\a(g_{(2,\a)})hS_{\bar{\a}}(B_\a(g_{(3,\a)})) B_\a(g_{(4,\a)})\ell S_{\bar{\a}}(B_\a(g_{(5,\a)})) \\
&=& g_{(1,\a)}B_\a(g_{(2,\a)})h\ell S_{\bar{\a}}(B_\a(g_{(3,\a)})) \\
&=& g\circ_B(h\ell).
\end{eqnarray*}
Hence $(H,\cdot,\circ_B)$ is a cocommutative $\pi$-Hopf brace.
\end{proof}

\section*{Acknowledgement}

This work was supported by the Shandong Provincial Natural Science Foundation (No. ZR2022QA007) and the NSF of Jining University (Nos. 2021ZYRC05, 2018BSZX01).

\end{document}